\documentclass[11pt,eqno]{article}
\usepackage{cite}
\usepackage{amssymb}

\usepackage{amsmath}
\allowdisplaybreaks[4]
\usepackage{array}
\usepackage{graphicx}
\usepackage{cases}
\usepackage{color}
\usepackage{multirow}
\usepackage{picinpar}
\usepackage{bm}

\usepackage[all]{xy}%

\definecolor{purple}{rgb}{0.5804,0.0000,0.8275}

\begin{document}
\today\\
\newcommand{\defi}{\stackrel{\Delta}{=}}
\newcommand{\qed}{\hphantom{.}\hfill $\Box$\medbreak}
\newcommand{\A}{{\cal A}}
\newcommand{\B}{{\cal B}}
\newcommand{\U}{{\cal U}}
\newcommand{\G}{{\cal G}}
\newcommand{\cZ}{{\cal Z}}
\newcommand{\proof}{\noindent{\bf Proof \ }}
\newcommand\one{\hbox{1\kern-2.4pt l }}
\newcommand{\Item}{\refstepcounter{Ictr}\item[(\theIctr)]}
\newcommand{\QQ}{\hphantom{MMMMMMM}}

\newtheorem{theorem}{Theorem}[section]
\newtheorem{Proposition}{Proposition}[section]
\newtheorem{condition}{Condition}[section]
\newtheorem{lemma}{Lemma}[section]
\newtheorem{corollary}{Corollary}[section]
\newtheorem{definition}{Definition}[section]
\newtheorem{remark}{Remark}[section]
\newtheorem{assumption}{Assumption}[section]
\newtheorem{example}{Example}[section]
\newenvironment{cproof}
{\begin{proof}
 [Proof.]
 \vspace{-3.2\parsep}}
{\renewcommand{\qed}{\hfill $\Diamond$} \end{proof}}
\newcommand{\erhao}{\fontsize{21pt}{\baselineskip}\selectfont}
\newcommand{\xiaoerhao}{\fontsize{18pt}{\baselineskip}\selectfont}
\newcommand{\sanhao}{\fontsize{15.75pt}{\baselineskip}\selectfont}
\newcommand{\sihao}{\fontsize{14pt}{\baselineskip}\selectfont}
\newcommand{\xiaosihao}{\fontsize{12pt}{\baselineskip}\selectfont}
\newcommand{\wuhao}{\fontsize{10.5pt}{\baselineskip}\selectfont}
\newcommand{\xiaowuhao}{\fontsize{9pt}{\baselineskip}\selectfont}
\newcommand{\liuhao}{\fontsize{7.875pt}{\baselineskip}\selectfont}
\newcommand{\qihao}{\fontsize{5.25pt}{\baselineskip}\selectfont}

\makeatletter
\newcommand{\figcaption}{\def\@captype{figure}\caption}
\newcommand{\tabcaption}{\def\@captype{table}\caption}
\makeatother

\newcounter{Ictr}

\renewcommand{\theequation}{
\arabic{equation}}
\renewcommand{\thefootnote}{\fnsymbol{footnote}}

\def\A{\mathcal{A}}

\def\C{\mathcal{C}}
\def\F{\mathcal{F}}
\def\V{\mathcal{V}}
\def\K{\mathcal{K}}

\def\I{\mathcal{I}}

\def\Y{\mathcal{Y}}

\def\X{\mathcal{X}}

\def\J{\mathcal{J}}

\def\Q{\mathcal{O}}

\def\W{\mathcal{W}}

\def\S{\mathcal{S}}

\def\T{\mathcal{T}}

\def\L{\mathcal{L}}

\def\M{\mathcal{M}}

\def\N{\mathcal{N}}
\def\R{\mathbb{R}}
\def\H{\mathbb{H}}


\begin{center}
\topskip10mm
\LARGE{\bf Low-Rank Matrix Optimization Over Affine Set}
\end{center}

\begin{center}
\renewcommand{\thefootnote}{\fnsymbol{footnote}}Xinrong Li$^{1}$,~Naihua Xiu$^{1}$,~Ziyan Luo$^{1}$
\footnote{
1. Department of Applied Mathematics,
Beijing Jiaotong University, Beijing 100044, P. R. China; X. Li (lixinrong0827@163.com), Z. Luo (starkeynature@hotmail.com, N. Xiu (nhxiu@bjtu.edu.cn).
}\\
{\small

}

\end{center}

\begin{abstract}
The low-rank matrix optimization with affine set (rank-MOA) is to minimize a continuously differentiable function over a low-rank set intersecting with an affine set. Under some suitable assumptions, the intersection rule of the Fr\'{e}chet normal cone to the feasible set of the rank-MOA is established. This allows us to propose the first-order optimality conditions for the rank-MOA in terms of the so-called F-stationary point and the $\alpha$-stationary point. Furthermore, the second-order optimality analysis, including the necessary condition and the sufficient one, is proposed as well. All these results will enrich the theory of low-rank matrix optimization and give potential clues to designing efficient numerical algorithms for seeking low-rank solutions. Meanwhile, we illustrate our proposed optimality analysis for several specific applications of the rank-MOA including the Hankel matrix approximation problem in system identification and the low-rank representation on linear manifold in signal processing.

\end{abstract}


\vskip12pt

\numberwithin{equation}{section}

\section{Introduction}

In this paper, we consider the following low-rank matrix optimization problem:
\begin{equation}\label{P}
  \begin{aligned}
\min\limits_{X\in\mathbb{R}^{m\times n}}&~~f(X)\\
s.t.&~~ \langle A^i, X \rangle=b_i,~i=1,\ldots,l\\\nonumber
 & ~~\mathrm{rank}(X)\leq r,\nonumber
  \end{aligned}\tag{rank-MOA}
\end{equation}
where $f: \mathbb{R}^{m\times n}\rightarrow \mathbb{R}$ is a continuously differentiable or twice continuously differentiable function, $A^i\in\mathbb{R}^{m\times n}$ and $b_i\in\mathbb{R}$ for each $i=1,\ldots,l$ are the given data. The
inner product is $\langle X,Y\rangle:=\sum_{ij} X_{ij}Y_{ij}$, and  $\rm{rank}(X)$ denotes the rank of $X\in\mathbb{R}^{m\times n}$. $r$ is a positive integer smaller than $n$. For convenience, we denote the low-rank affine set of the \ref{P} as $\mathcal{L}\cap\mathbb{R}(r)$ with
  $$\mathcal{L}:=\{X\in\mathbb{R}^{m\times n}: \langle A^i, X \rangle=b_i, ~i=1,\ldots,l\},~~\mathbb{R}(r):=\{X\in\mathbb{R}^{m\times n}:\mathrm{rank}(X)\leq r\}.$$
The problems with embedded low-rank structures arise in diverse areas such as system identification \cite{liu2009interior-point,Defeng2013hankel}, control\cite{SeogStructurally}, signal processing\cite{wright2010sparse,xue2018robust}, collaborative filtering\cite{gillis2011low-rank}, statistics\cite{wainwright2014structured,she2017robust}, finance\cite{42005Rank,QiA2006}, machine learning\cite{kobayashi2014low-rank,xing2002distance}, among others. Due to the low-rank constraint, however, low-rank matrix optimizations of the form \ref{P} are highly nonconvex and computationally NP-hard in general\cite{fazel2002matrix}. In order to deal with the rank constraint and to find a low-rank solution, this problem has attracted a lot of research attention over the last few years.

Minimizing or penalizing the convex (nonconvex) relaxation of the rank function has proven to be an effective method for generally keeping its rank small, and a vast amount of recent work has focused on this technique\cite{fazel2002matrix,  fazel2003, Fornasier2011Low, Mohan2012Iterative,Lai2013Improved}; however, many problems require finding a matrix whose rank is constrained to be a particular value. Handling rank constraint is known to be difficult since the rank is discontinuous and nonconvex. This has motivated several researchers to find equivalent representations for low-rank constraint.  We refer the interested reader to the papers \cite{burer2005local,JournLow, Gao2010STRUCTURED,delgado2016novel,zhou2018fast} and the references therein.  However, little research has been done in optimality theory for the original low-rank matrix optimization merely with rank constraint, let alone the \ref{P}.

The analysis of optimality conditions is one of the most important topics to solve the original low-rank matrix optimization. This is often done through the tangent and normal cones to the constraint set of the problem. During the past few years, various notions of tangent and normal cones have been introduced to deal with the low-rank constraint directly. For the matrix optimization with a single low-rank constraint, the Bouligand tangent cone to the low-rank set has been derived in \cite{schneider2015convergence}, the proximal and Mordukhovich normal cones to the low-rank set have been given in \cite{luke2013prox-regularity}, and the Clarke tangent cone and its corresponding normal cone to the low-rank set have been presented in \cite{Li2018Social}. Further study on the low-rank matrix optimization with some additional simple constraints have also been discussed, including the Bouligand tangent cone and its corresponding normal cone to the intersection of the low-rank set and a unit ball in \cite{Cason2013Iterative}, the Mordukhovich normal cone to the intersection of the low rank set and the cone of positive semidefinite matrices in \cite{tam2017regularity}, and the Fr\'{e}chet normal cones to the intersection of the low-rank set and some given spectral sets in \cite{Li2019Social}. All the involved constraint sets are of symmetry in the spectral sense. We generalize this line of work by focusing on the low-rank matrix optimization problem whose feasible set is the intersection of the low-rank set and an affine set.

When some affine constraints are involved in the low-rank matrix optimization problem, in order to ensure that  optimality conditions are established, a constraint qualification (CQ) is needed. It is well-known that a CQ is a condition imposed on constraint functions so that Karush-Kuhn-Tucker points(called stationary points in Section 3) hold at a local minimizer. There exist very weak constraint qualifications such as Guignard's and Abadie's constraint qualifications \cite{guignard1969generalized,Abadie1967} but they are not easy to verify since it involves computing the tangent or normal cone of the constraint region. The challenge is to find verifiable constraint qualifications that are applicable to the \ref{P}.

In this paper,  we tackle the \ref{P} directly and study optimality conditions tailored for the \ref{P}. To achieve this goal, we discuss the intersection rule of the Fr{\'e}chet normal cone by the linear independence assumptions to ensure that a local minimizer satisfies the stationary points. We also define two types of stationary points for the \ref{P}, which are fundamental for the development of algorithm design. Then, we obtain the first-order and the second-order optimality conditions associated with the stationary points and global/local minimizers for the \ref{P} under suitable conditions.  Moreover, we illustrate how to apply our results to the problems of Hankel matrix approximation and low-rank representation on linear manifold.

This paper is organized as follows. In Section 2, we review some related concepts and properties for normal cones and  establish the intersection rule of Fr\'{e}chet normal cone for the feasible set. In Section 3, we introduce two kinds of stationary points and investigate the first-order and second-order optimality conditions for the \ref{P}. In Section 4, we discuss two important applications of rank-MOA to illustrate our proposed optimality analysis. Conclusions are made in Section 5.

Our notation is standard. Let $\mathbb{R}^{m\times n}$ be the Euclidean space of the real $m \times n$ matrices equipped with the inner product $\langle X,Y\rangle=\sum_{ij} X_{ij}Y_{ij}$ and the induced Frobenius norm $\|X\|_F$. For any $X\in\mathbb{R}^{m\times n}$, $\|X\|_2$ denotes the spectral norm, i.e., the largest singular value of $X$.
We denote by $X_{ij}$ the $(i,j)$-th entry of $X$. We use $x_{j}$ to represent the $j$th column of $X$, $j=1,\ldots, n$. Let $J\subseteq\{1,\ldots, n\}$ be an index set. $|J|$ is the cardinality of $J$. 
We use $X_J$ to denote the sub-matrix of $X$ that contains all columns indexed by $J$.
$\mathcal{O}^p$ is the set of all $p\times p$ orthogonal matrices, i.e., $\mathcal{O}^p=\{A\in\mathbb{R}^{p\times p}~|~A ^\top  A =I_p\},$  where $I_p$  denotes the $p$ order identity matrix. 
Let $\mathbb{R}^n$ be the Euclidean space.
For a vector $x\in\mathbb{R}^n$, let $\mathrm{Diag}(x)$ be an $n\times n$ diagonal matrix with diagonal entries $x_i$.
Denote $x_J = (x_i)_J\in\mathbb{R}^{|J|}$ as the subvector of $x$ corresponding to the indices in $J$.
$\|\cdot\|_0$ is the $l_0$ norm counting the number of nonzero entries of $x$.
For a nonempty and closed set $\Omega\subset\mathbb{R}^{m\times n}$, the projection of a matrix $X\in\mathbb{R}^{m\times n}$ onto $\Omega$ is defined as $\Pi_\Omega(X):=\rm{argmin}_{Z\in \Omega}\|Z-X\|_F$.

\section{Normal Cone Intersection Rule}
\label{sec:rule}
Normal cones to feasible sets play an important role in the optimality analysis of constrained optimization. As the feasible set of the \ref{P} is an intersection of the affine set ${\mathcal{L}}$ and the low-rank matrix set $\mathbb{R}(r)$, we will discuss the intersection rule of the normal cone to such an intersection set in this section. Before proceeding, several related concepts and properties related to normal cones are reviewed, mainly followed from the classical monograph \cite{Rockafellar2013Variational}.

Recall that a set $\mathcal{K}$ is called a cone, if $\gamma\mathcal{K}\subseteq \mathcal{K}$ holds for all $\gamma\geq 0$. The polar of the cone $\mathcal{K}$ is, denoted as $\mathcal{K}^\circ$, is defined by $\mathcal{K}^\circ=\{Y| \langle Y, X\rangle\leq 0,~ \forall X\in\mathcal{K}\}$.
If $\mathcal{K}_1$ and $\mathcal{K}_2$ are nonempty cones in $\mathbb{R}^{m\times n}$, we have
\begin{equation}\label{01}
(\mathcal{K}_1\cup \mathcal{K}_2)^\circ=(\mathcal{K}_1+\mathcal{K}_2)^\circ=\mathcal{K}_1^\circ\cap \mathcal{K}_2^\circ.
\end{equation}
Furthermore, if $\mathcal{K}_1$ and $\mathcal{K}_2$ are closed convex cones, then
\begin{equation}\label{02}
(\mathcal{K}_1\cap \mathcal{K}_2)^\circ=\mathcal{K}_1^\circ+ \mathcal{K}_2^\circ.
\end{equation}
For any given nonempty, closed set $\Omega \subseteq \mathbb{R}^{m\times n}$, and any $X\in \Omega$, the Bouligand tangent cone and its polar (also called the Fr\'{e}chet normal cone) to $\Omega$ at $X$, termed as $\mathrm{T}_\Omega^B(X)$ and $\mathrm{N}_\Omega^F(X)$, are defined by
\begin{eqnarray*}
\mathrm{T}_\Omega^B(X):&=&\left \{\Xi\in\mathbb{R}^{m\times n}: \begin{array}{l}\exists\{X^k\}\subseteq\Omega~\text{with}~X^k\rightarrow X;~\exists\{t_k\}~\text{with}\\  t_k\rightarrow0,~\mathrm{s.t.}~t_k^{-1}(X^k-X)\rightarrow\Xi,~\forall k\in\mathbb{N} \end{array}\right \},\\
\mathrm{N}_\Omega^F(X):&=&[\mathrm{T}_\Omega^B(X)]^\circ.
\end{eqnarray*}
Additionally, the Mordukhovich normal cone to $\Omega$ at $X$, termed as $\mathrm{N}_\Omega^M(X)$, is defined by
$$ \mathrm{N}_\Omega^M(X):=\limsup_{X'\xrightarrow{\Omega}X}\mathrm{N}_{\Omega}^F(X'),$$
where $X'\xrightarrow{\Omega}X$ means that $X'\in \Omega$ and $X'\rightarrow X$. It is seen that $\mathrm{N}_{\Omega}^F(X)\subseteq \mathrm{N}_\Omega^M(X)$.
\begin{definition}(See \cite[Definition 6.4]{Rockafellar2013Variational})
The set $\Omega$ being locally closed at $X\in\Omega$ and satisfying $\mathrm{N}_{\Omega}^F(X)=\mathrm{N}_\Omega^M(X)$ is called regular at $X$ in the sense of Clarke.
\end{definition}

 Particularly, if $\Omega$ is a closed convex set or a smooth submanifold, then $\Omega$ is regular. In this case, for any given $X\in \Omega$, we simply write
\begin{eqnarray*}
\mathrm{T}_\Omega(X):=\mathrm{T}_\Omega^B(X),~\text{and}~\mathrm{N}_\Omega(X):=\mathrm{N}_{\Omega}^F(X)=\mathrm{N}_\Omega^M(X).
\end{eqnarray*}

Next we concern with the calculations of tangent cones and normal cones to the union  and intersection of finitely many nonempty closed sets $\Omega_1,\ldots, \Omega_k$. For $X\in \bigcup^k_{i=1}\Omega_i$, we have
 \begin{equation}\label{03}
\mathrm{T}^B_{\bigcup^k_{i=1}\Omega_i}(X)=\bigcup^k_{i=1}\mathrm{T}^B_{\Omega_i}(X).
 \end{equation}
Let $X\in \Omega_1\cap\Omega_2$. It holds that
\begin{equation}\label{04}
\mathrm{T}^B_{\Omega_1\cap\Omega_2}(X)\subseteq \mathrm{T}^B_{\Omega_1}(X)\cap \mathrm{T}^B_{\Omega_2}(X),~~~~\mathrm{N}^F_{\Omega_1\cap\Omega_2}(X)\supseteq \mathrm{N}^F_{\Omega_1}(X)+\mathrm{N}^F_{\Omega_2}(X).
\end{equation}
Under the basic qualification condition $\mathrm{N}^M_{\Omega_1}(X)\cap(-\mathrm{N}^M_{\Omega_2}(X))=\{0\}$, we also have
\begin{equation}\label{05}
\mathrm{N}^M_{\Omega_1\cap\Omega_2}(X)\subseteq \mathrm{N}^M_{\Omega_1}(X)+\mathrm{N}^M_{\Omega_2}(X).
\end{equation}
If in addition $\Omega_1$ and $\Omega_2$ are regular at $X$, then $\Omega_1\cap \Omega_2$ is regular at $X$ and
\begin{equation}\label{06}
 \mathrm{T}^B_{\Omega_1\cap\Omega_2}(X)=\mathrm{T}^B_{\Omega_1}(X)\cap \mathrm{T}^B_{\Omega_2}(X),~~~~\mathrm{N}^F_{\Omega_1\cap\Omega_2}(X)=\mathrm{N}^F_{\Omega_1}(X)+\mathrm{N}^F_{\Omega_2}(X).
\end{equation}

Recall that for any matrix $X \in \mathbb{R}^{m\times n}$ (without loss of generality, assume $m\geq n$) of rank $s:=\mathrm{rank}(X)$, its singular value decomposition (SVD) can be expressed as
\begin{equation}\label{SVD}
X =U \Sigma(X ){V} ^\top = [U_{\Gamma}~ U_{\Gamma_m^\perp} ]
     \left[
      \begin{array}{cc}
       {\Sigma}_{\Gamma\Gamma }(X ) & 0 \\
       0 & 0 \\
     \end{array}
   \right][ V_{\Gamma } ~ {V}_{\Gamma_n^\perp} ] ^\top
 \end{equation}
where $U\in\mathcal{O}^m$ and $V\in\mathcal{O}^n$, $\Gamma\subseteq \{1, \ldots, n\}$ is the index set for nonzero singular values with $|\Gamma|=s$, $\Sigma_{\Gamma\Gamma}(X)\in \mathbb{R}^{s\times s}$ is the submatrix of the diagonal matrix $\Sigma(X)$ indexed by $\Gamma$, $\Gamma_m^\bot = \{1,\ldots, m\}\setminus \Gamma$, and $\Gamma_n^\bot = \{1,\ldots, n\}\setminus \Gamma$. Denote $\mathbb{R}^{m\times n}$ by $\mathbb{R}_{s}:=\{X\in\mathbb{R}^{m\times n}: \mathrm{rank}(X)=s\}$.
It is well known that $\mathbb{R}_{s}$ is a smooth submanifold (see, \cite{Helmke1995Critical}), and hence $\mathbb{R}_{s}$ is regular. The corresponding tangent and normal cones (spaces) to $\mathbb{R}_{s}$ at $X$ have the following explicit formulae,
\begin{eqnarray*}
\mathrm{T}_{\mathbb{R}_{s}}(X )&=&\left\{H\in\mathbb{R}^{m\times n}: U_{\Gamma_m^\perp}^\top HV_{\Gamma_n^\perp}=0 \right\},\\
\mathrm{N}_{\mathbb{R}_{s}}(X )&=&\left\{{U}_{\Gamma_m^\perp}D{V}_{\Gamma_n^\perp}^\top\in \mathbb{R}^{m\times n}: D\in\mathbb{R}^{(m-s)\times (n-s)}\right\}.
\end{eqnarray*}

With the aid of the above expressions of tangent and normal cones to the fixed-rank matrix set $\mathbb{R}_s$, the explicit formulae for the Bouligand tangent cone, the Fr\'{e}chet normal cone, and the Mordukhovich normal cone to the low-rank matrix set $\mathbb{R}(r)$, have been characterized in \cite[Theorem 3.2]{schneider2015convergence} and \cite[Proposition 3.6]{luke2013prox-regularity}, respectively, as the following lemma stated.

\begin{lemma}\label{NBF}
For any given $X \in\mathbb{R}(r)$ of rank $s$, we have
 \begin{eqnarray*}
\mathrm{T}^B_{\mathbb{R}(r)}(X )&=&\mathrm{T}_{\mathbb{R}_{s}}(X )+\{\Xi\in \mathrm{N}_{\mathbb{R}_{s}}(X ): \mathrm{rank}(\Xi)\leq r-s\},\\
\mathrm{N}^F_{\mathbb{R}(r)}(X )&=&\begin{cases}
\mathrm{N}_{\mathbb{R}_{s}}(X ),& \text{if}~s=r,\\
\{0\},&\text{if}~s<r.
\end{cases}\\
\mathrm{N}^M_{\mathbb{R}(r)}(X)&=&\{W\in \mathrm{N}_{\mathbb{R}_{s}}(X): \text{rank}(W)\leq n-r\}.
\end{eqnarray*}
\end{lemma}

We are now in a position to study the intersection rule of Fr{\'e}chet normal cone to the feasible set of the \ref{P}. For $X\in \mathcal{L}\cap \mathbb{R}(r)$ with its SVD as in \eqref{SVD}, we denote $\mathcal{J}:=\{J\subseteq\{1,\ldots, n\}:|J|=r, \Gamma\subseteq J\}$ and introduce the following subset of the low-rank set with respect to $X$ together with the matrices $U$ and $V$ in \eqref{SVD}:
\begin{equation}\label{SX}
\mathbb{R}_{(X,U,V)}(r):=\bigcup_{J\in\mathcal{J}}\mathbb{R}_{(X,U,V)}(J)\subseteq \mathbb{R}(r),
\end{equation}
 where
\begin{equation}\label{SJX}
\mathbb{R}_{(X,U,V)}(J):=\left\{U_{J} A V_{J}^\top: A\in \mathbb{R}^{r\times r}\right\}
\end{equation}
is a subspace associated with $(X,U,V)$. For simplicity, we use $\mathbb{R}_X(r)$ and $\mathbb{R}_X(J)$ to briefly denote $\mathbb{R}_{(X,U,V)}(r)$ and $\mathbb{R}_{(X,U,V)}(J)$, respectively. In particular, if $\text{rank}(X)=r$, we have
$$\mathbb{R}_X(r)=\mathbb{R}_X(\Gamma) = \left\{U_{\Gamma} A V_{\Gamma}^\top: A\in \mathbb{R}^{r\times r}\right\}.$$

\begin{lemma}\label{NNN}
Let $X \in \mathbb{R}(r)$ be a rank $s$ matrix with its SVD as in \eqref{SVD}, and $\mathbb{R}_X(J)$ and $\mathbb{R}_X(r)$ be defined as in \eqref{SJX} and \eqref{SX}. The following statements hold.
\begin{itemize}
\item[(i)] For any $J\in \mathcal{J}$,
\begin{equation}\label{TNspace}
\mathrm{N}_{\mathbb{R}_X(J)}(X )
=\{W\in\mathbb{R}^{m\times n}:U_J^\top W V_J=0\}.\end{equation}
\item[(ii)] Let $[m-n] =\{1,\ldots,m\}\setminus \{1,\ldots, n\}$. Then
\begin{equation}\label{TNspace1}\mathrm{N}^F_{\mathbb{R}_X(r)}(X)=\left\{ \begin{array}{ll}
           \left\{
  W\in\mathbb{R}^{m\times n}:U_\Gamma^\top W V_\Gamma=0 \right\}, & \text{if}~s=r\\
         \left\{ \begin{array}{l}
         U_{\Gamma_n^\bot}D {V}_{\Gamma_n^\bot}^\top+\\U_{[m-n]}HV^\top
         \end{array}:\begin{array}{l}
                                                                               \text{diag}(D)=0\\
                                                                               D\in\mathbb{R}^{(n-s)\times (n-s)} \\
                                                                                 H\in\mathbb{R}^{(m-n)\times n}
                                                                             \end{array}
         \right\}, &\text{if}~s=r-1,\\
         \left\{U_{[m-n]}HV^\top: H\in\mathbb{R}^{(m-n)\times n}\right\},&\text{if} ~s\leq r-2.\\
             \end{array}
\right.
\end{equation}
\end{itemize}
\end{lemma}
\begin{proof} 
The first part follows readily from the definition of the subspace $\mathbb{R}_X(J)$. Note that
 \begin{eqnarray}\label{J0}
\mathrm{N}^F_{\mathbb{R}_X(r)}(X )&=&\left(\mathrm{T}^B_{\mathbb{R}_X(r)}(X )\right)^\circ
=\left(\mathrm{T}^B_{\bigcup \limits_{J\in\mathcal{J} }\mathbb{R}_X(J)}(X )\right)^\circ=\left(\bigcup \limits_{J\in\mathcal{J} }\mathrm{T}_{\mathbb{R}_X(J)}(X )\right)^\circ\nonumber\\
&=&\bigcap \limits_{J\in\mathcal{J} }\mathrm{N}_{\mathbb{R}_X(J)}(X )\nonumber\\
&=&\bigcap \limits_{J\in\mathcal{J} }\{W\in\mathbb{R}^{m\times n}: U_J^\top WV_J=0\}\nonumber\\
&=& \{W\in\mathbb{R}^{m\times n}: U_J^\top WV_J=0,~\forall J\in\mathcal{J}\}.
\end{eqnarray}
Specifically, if $s=r$, then $\mathcal{J}=\{\Gamma\}$, and hence $\mathrm{N}_{\mathbb{R}_X(r)}(X)=\{W\in\mathbb{R}^{m\times n}:U_\Gamma^\top WV_\Gamma=0 \}$. If $s=r-1$, then $\mathcal{J}=\{\Gamma\cup \{i\}:i\in\Gamma_n^\bot\}$. Thus
\begin{eqnarray*}
\mathrm{N}_{\mathbb{R}_X(r)}(X)&=&\bigcap \limits_{i\in\Gamma_n^\bot}\left \{W\in\mathbb{R}^{m\times n}: \begin{array}{l}
U_\Gamma^\top WV_\Gamma=0,~u_{i}^\top Wv_{i}=0,\\  u_{i}^\top WV_{\Gamma}=0,~U_{\Gamma}^\top Wv_{i}=0\end{array}\right \}\\
&=&  \left \{W\in\mathbb{R}^{m\times n}: \begin{array}{l}
U_\Gamma^\top WV_\Gamma=0,~\text{diag}(U_{\Gamma_n^\bot}^\top WV_{\Gamma_n^\bot})=0\\  U_{\Gamma_n^\bot}^\top WV_{\Gamma}=0,~U_{\Gamma}^\top WV_{\Gamma_n^\bot}=0\end{array}\right \}\\
&=& \left\{U_{\Gamma_n^\bot}D {V}_{\Gamma_n^\bot}^\top+U_{[m-n]}HV^\top:\begin{array}{l}
                                                                               \text{diag}(D)=0\\
                                                                               D\in\mathbb{R}^{(n-s)\times (n-s)} \\
                                                                                 H\in\mathbb{R}^{(m-n)\times n}
                                                                             \end{array}\right\}.
\end{eqnarray*}
If $s\leq r-2$. Consider any given $W\in\mathrm{N}_{\mathbb{R}_X(r)}(X)$. For any $i,j\in\{1,2,\ldots,n\}$, there exists an index set $J\subseteq\mathcal{J}$ such that $\{i,j\}\subset J$. The fact $U_J^\top WV_J=0$ in \eqref{J0} implies that $u_i^TWv_j=0$.
From the arbitrariness of $i$ and $j$ in $\{1,2,\ldots,n\}$, we have $U_{\{1,\ldots,n\}}^\top WV=0$,
which indicates that $W=U_{[m-n]}HV^\top$ for some $H\in \mathbb{R}^{(m-n)\times n}$. This completes the proof.
\end{proof}

Let $X\in \mathbb{R}^{m\times n}$ with its SVD as in \eqref{SVD}. For any given matrices $A^1$, $\ldots$, $A^l\in\mathbb{R}^{m\times n}$. We denote
\begin{equation}\label{Ti}
 T^i_X=
     \left[
      \begin{array}{cc}
      U_{\Gamma} ^\top  A^iV_{\Gamma}&   U_{\Gamma} ^\top  A^i V_{{\Gamma}_n^\bot} \\
       U_{{\Gamma }_m^\bot} ^\top  A^iV_{\Gamma }& 0 \\
     \end{array}
   \right],~~~~~
R^i_X=
     \left[
      \begin{array}{cc}
      U_{\Gamma } ^\top  A^iV_{\Gamma }&   0 \\
       0& 0 \\
     \end{array}
   \right]
\end{equation}
for $i=1,\ldots,l$.
Introduce the following two assumptions.

\begin{assumption}\label{ass1}
The matrices $T^i_X$, $i=1,\ldots, l$, are linearly independent.
\end{assumption}

\begin{assumption}\label{ass2}
The matrices $R^i_X$, $i=1,\ldots, l$, are linearly independent.
\end{assumption}

Here, Assumption \ref{ass1} is called the primal nondegeneracy condition in \cite[Definition 5]{Alizadeh1997Complementarity} in the context of semidefinite programming, and  Assumption \ref{ass2} is a stronger variant of Assumption \ref{ass1}. Let $X$ be a feasible point of the problem \ref{P} with $\text{rank}(X)=s$. By the discussion of \cite[Section 5, Page 480]{bonnans2013perturbation}, we have that Assumption \ref{ass1} can happen only if  $l\leq mn-(m-s)(n-s)$. Similarly, a necessary condition for Assumption \ref{ass2} holding is $l\leq s^2$. Based on these two assumptions, we have the following intersection rule of normal cone.

\begin{Proposition}\label{Assu}
 For any given $X\in\mathcal{L}\cap\mathbb{R}(r)$ with its SVD as in \eqref{SVD}, and any index set $J$ satisfying $\Gamma \subseteq J$, we have
\begin{itemize}
\item[(i)]  If  Assumption \ref{ass1} holds at $X$, then $\mathrm{N}^M_{\mathbb{R}(r)}(X )\cap \mathrm{N}_{\mathcal{L}}(X )=\{0\}$;
\item[ (ii)]  If  Assumption \ref{ass2} holds at $X$, then $\mathrm{N}_{\mathbb{R}_X(J)}(X )\cap \mathrm{N}_{\mathcal{L}}(X )=\{0\}$.
\end{itemize}
\end{Proposition}
\begin{proof} 
(i) By virtue of Theorem 6 in \cite {Alizadeh1997Complementarity}, Assumption~1 holds at $X $ if and only if $\mathrm{N}_{\mathbb{R}_s}(X )\cap \mathrm{N}_{\mathcal{L}}(X )=\{0\}$.  The desired assertion in (i) then follows from the fact $\mathrm{N}^M_{\mathbb{R}(r)}(X )\subseteq \mathrm{N}_{\mathbb{R}_s}(X )$.

(ii) Assume on the contrary that there exists a nonzero matrix $W\in \mathrm{N}_{\mathbb{R}_X(J)}(X )\cap \mathrm{N}_{\mathcal{L}}(X )$, that is, there exist $t^i\in\mathbb{R}$ ($i=1,\ldots, l$) not all zero such that $\sum\limits_{i=1}^l t^iA^i \in  \mathrm{N}_{\mathbb{R}_X(J)}(X )$. From \eqref{TNspace},  we get $$ U_{J} ^\top  \sum\limits_{i=1}^l t^iA^iV_{J}=0,$$
 which implies  that $ U_{\Gamma } ^\top  \sum\limits_{i=1}^lt^iA^iV_{\Gamma }=0$. Thus
\begin{eqnarray*}
\sum_{i=1}^lt^iR_X^i=
     \left[
      \begin{array}{cc}
     U_{\Gamma } ^\top  \sum\limits_{i=1}^l t^iA^iV_{\Gamma }&   0 \\
      0& 0 \\
     \end{array}
   \right]
=0,
   \end{eqnarray*}
   which contradicts to the linear independence of $R_X^i$'s in Assumption \ref{ass2}. Thus, we have $\mathrm{N}_{\mathbb{R}_X(J)}(X )\cap \mathrm{N}_{\mathcal{L}}(X )=\{0\}$. This completes the proof.
 \end{proof}

\begin{lemma}\label{TBeq}
Let $X \in\mathcal{L}\cap\mathbb{R}(r)$ with its SVD as in \eqref{SVD}. If Assumption \ref{ass2} holds at $X$, then
\begin{equation}\label{Nequ}
\mathrm{T}^B_{\mathcal{L}\cap \mathbb{R}_X(r)}(X )=\mathrm{T}_{\mathcal{L}}(X )\cap \mathrm{T}^B_{\mathbb{R}_X(r)}(X ).
\end{equation}
\end{lemma}
\begin{proof}
Note that $\mathcal{L}$ and $\mathbb{R}_X(J)$ are regular at $X$. Since Assumption \ref{ass2} holds at $X$, from Proposition \ref{Assu} (ii) and \eqref{06}, we obtain that
$$\mathrm{T}_{\mathcal{L}\cap\mathbb{R}_X(J)}(X )=\mathrm{T}_{\mathcal{L}}(X )\cap \mathrm{T}_{\mathbb{R}_X(J)}(X ).$$
This together with \eqref{03} yields 
\begin{eqnarray}\label{T}
\mathrm{T}^B_{\mathcal{L}\cap\mathbb{R}_X(r)}(X )&=&\bigcup \limits_{J\in\mathcal{J} }\mathrm{T}_{\mathcal{L}\cap\mathbb{R}_X(J)}(X )
=\bigcup \limits_{J\in\mathcal{J} }\left(\mathrm{T}_{\mathcal{L}}(X )\cap \mathrm{T}_{\mathbb{R}_X(J)}(X )\right)\\\nonumber
&=&\mathrm{T}_{\mathcal{L}}(X )\cap \left(\bigcup \limits_{J\in\mathcal{J} }\mathrm{T}_{\mathbb{R}_X(J)}(X )\right)\\\nonumber
&=&\mathrm{T}_{\mathcal{L}}(X )\cap \mathrm{T}^B_{\bigcup \limits_{J\in\mathcal{J} }\mathbb{R}_X(J)}(X )\\\nonumber
&=&\mathrm{T}_{\mathcal{L}}(X )\cap \mathrm{T}^B_{\mathbb{R}_X(r)}(X ).
\end{eqnarray}
This completes the proof.
\end{proof}

Based on the above result, we investigate the intersection rule of the Fr\'{e}chet normal cone to $\mathcal{L}\cap\mathbb{R}(r)$.

\begin{theorem}\label{NFeq}
Let $X \in\mathcal{L}\cap\mathbb{R}(r)$ with $\mathrm{rank}(X )=s$.
\begin{itemize}
\item[(i)] If $s=r$ and Assumption \ref{ass1} holds at $X$, then
\begin{equation}\label{Nequ1}
\mathrm{N}^F_{\mathcal{L}\cap\mathbb{R}(r)}(X )=\mathrm{N}_{\mathcal{L}}(X )+\mathrm{N}^F_{\mathbb{R}(r)}(X ).
\end{equation}
\item[(ii)] If $s<r$ and Assumption \ref{ass2} holds at $X$, then
\begin{equation}\label{Nequ2}
\mathrm{N}^F_{\mathcal{L}\cap\mathbb{R}(r)}(X)\subseteq \mathrm{N}^F_{\mathcal{L}\cap\mathbb{R}_X(r)}(X)=\mathrm{N}_{\mathcal{L}}(X)+\mathrm{N}^F_{\mathbb{R}_X(r)}(X ).
\end{equation}

If, in addition, $\mathrm{N}^F_{\mathbb{R}_X(r)}(X )\subseteq  \mathrm{T}_{\mathcal{L}}(X)$(this is true in the case of $s\leq r-2$ and $m=n$), then
\begin{equation}\label{Nequ3}
\mathrm{N}^F_{\mathcal{L}\cap\mathbb{R}(r)}(X)=\mathrm{N}_{\mathcal{L}}(X )+\mathrm{N}^F_{\mathbb{R}(r)}(X ) = \mathrm{N}_{\mathcal{L}}(X ).
\end{equation}
\end{itemize}
\end{theorem}
\begin{proof}
(i) If $s=r$, it is known from Lemma \ref{NBF} that $\mathrm{N}^F_{\mathbb{R}(r)}(X ) = \mathrm{N}_{\mathbb{R}^s}(X)= \mathrm{N}^M_{\mathbb{R}(r)}(X )$. Thus, in this case, $\mathbb{R}(r)$ is regular at $X$. Together with the regularity of the convex set $\mathcal{L}$ at $X$, we obtain that
$\mathcal{L}\cap\mathbb{R}(r)$ is also regular at $X$, i.e., $\mathrm{N}^F_{\mathcal{L}\cap\mathbb{R}(r)}(X) = \mathrm{N}^M_{\mathcal{L}\cap\mathbb{R}(r)}(X)$. By utilizing (i) of Proposition \ref{Assu}, Assumption \ref{ass1} ensures that
$\mathrm{N}^M_{\mathbb{R}(r)}(X)\cap \left(-\mathrm{N}_{\mathcal{L}}(X)\right) = \mathrm{N}^M_{\mathbb{R}(r)}(X)\cap \mathrm{N}_{\mathcal{L}}(X) =\{0\}$. Thus, $\mathrm{N}^F_{\mathcal{L}\cap\mathbb{R}(r)}(X)\subseteq \mathrm{N}^F_{\mathbb{R}(r)}(X) + \mathrm{N}_{\mathcal{L}}(X)$. Combining with the second inclusion in \eqref{04}, the desired assertion is obtained.

(ii) The first inclusion in  \eqref{Nequ2} follows readily from \eqref{SX}. For the remaining equality, by virtue of the second inclusion in \eqref{04}, it suffices to show
$$\mathrm{N}^F_{\mathcal{L}\cap\mathbb{R}_X(r)}(X)\subseteq \mathrm{N}_{\mathcal{L}}(X )+\mathrm{N}^F_{\mathbb{R}_X(r)}(X ).$$ From Lemma \ref{TBeq}, \eqref{01}  and \eqref{02}, we obtain that
\begin{eqnarray}\label{N}
\mathrm{N}^F_{\mathcal{L}\cap\mathbb{R}_X(r)}(X )&=&\left(\mathrm{T}^B_{\mathcal{L}\cap\mathbb{R}_X(r)}(X)\right)^\circ
=\left(\bigcup \limits_{J\in\mathcal{J} }\mathrm{T}^B_{\mathcal{L}}(X )\cap \mathrm{T}^B_{\mathbb{R}_X(J)}(X )\right)^\circ\\\nonumber
&=&\bigcap \limits_{J\in\mathcal{J} }\left(\mathrm{T}^B_{\mathcal{L}}(X )\cap \mathrm{T}^B_{\mathbb{R}_X(J)}(X ) \right)^\circ\\\nonumber
&=&\bigcap \limits_{J\in\mathcal{J} }\left(\mathrm{N}_{\mathcal{L}}(X )+ \mathrm{N}_{\mathbb{R}_X(J)}(X ) \right).
\end{eqnarray}
For any $H\in \mathrm{N}^F_{\mathcal{L}\cap\mathbb{R}_X(r)}(X )$, we have $H\in \mathrm{N}_{\mathcal{L}}(X )+ \mathrm{N}^F_{\mathbb{R}_{X}(J)}(X )$, for any $J\in\mathcal{J}$, that is, there exist $t^i(J)\in\mathbb{R}$,  $W(J)\in \mathrm{N}_{\mathbb{R}_X(J)}(X ),$
  such that
  \begin{equation}\label{HW}
  H=\sum_{i=1}^l t^i(J)A^i+W(J), ~\forall~J\in\mathcal{J}.
   \end{equation}
Note that for each $J\in \mathcal{J}$, it holds that $U_{\Gamma}^\top W(J)V_{\Gamma}=0$.
  Pre- and post-multiplying both sides of the equation \eqref{HW} by $U_{\Gamma}^\top$ and $V_{\Gamma}$, respectively, we obtain
$$U_{\Gamma}^\top HV_{\Gamma}=\sum_{i=1}^l t^i(J)U_{\Gamma}^\top A^i V_{\Gamma}, ~\forall~J\in\mathcal{J}.$$
For any distinct index set $J_0\in\mathcal{J}$, we can also get that
  $$U_{\Gamma}^\top HV_{\Gamma}=\sum_{i=1}^l t^i(J_0)U_{\Gamma}^\top A^i V_{\Gamma},  ~\forall~J_0(\neq J)\in\mathcal{J}.$$
  Invoking the linear independence in Assumption \ref{ass2}, we have
  $$t^i(J)=t^i(J_0)=:t^i, ~\forall~J,J_0\in\mathcal{J}.$$
  This implies that
  $$H-\sum_{i=1}^l t^iA^i\in\mathrm{N}_{\mathbb{R}_X(J)}(X),~\forall~J\in\mathcal{J},$$
  i.e.,
  $$H-\sum_{i=1}^l t^iA^i\in \bigcap \limits_{J\in\mathcal{J}} \mathrm{N}_{\mathbb{R}_X(J)}(X)= \mathrm{N}^F_{\mathbb{R}_X(r)}(X).$$
  Thus, $H\in \mathrm{N}_{\mathcal{L}}(X)+\mathrm{N}^F_{\mathbb{R}_X(r)}(X)$, which indicates that $\mathrm{N}^F_{\mathcal{L}\cap\mathbb{R}(r)}(X)\subseteq \mathrm{N}_{\mathcal{L}}(X )+\mathrm{N}^F_{\mathbb{R}_X(r)}(X )$.
This yields the assertions in \eqref{Nequ2}.

To show the remaining part, it is sufficient to show $\mathrm{N}^F_{\mathcal{L}\cap\mathbb{R}(r)}(X )\subseteq \mathrm{N}_{\mathcal{L}}(X)$ due to the second inclusion in \eqref{04} and the fact $\mathrm{N}^F_{\mathbb{R}(r)}(X ) = \{0\}$ for $s<r$. Consider any given $W\in \mathrm{N}^F_{\mathcal{L}\cap\mathbb{R}(r)}(X )$. As one can see from \eqref{Nequ2}, there exist $W_1\in \mathrm{N}_{\mathcal{L}}(X )$ and $W_2 \in \mathrm{N}^F_{\mathbb{R}_X(r)}(X )$ such that $W= W_1+W_2$. It then suffices to show $W_2=0$. This is trivial for the case of $m=n$ with $s \leq r-2$ since in this case $\mathrm{N}^F_{\mathbb{R}_X(r)}(X ) =\{0\}$ as stated in \eqref{TNspace1} in Lemma \ref{NNN}. Now we consider the case $m>n$ in two  cases.

\noindent(a) If $s = r-1$, we assume on the contrary that $W_2\neq 0$. By invoking the characterization in \eqref{TNspace1} in Lemma \ref{NNN}, there exist $D\in \mathbb{R}^{(n-s)\times (n-s)}$ with $diag(D) =0$, and $H\in \mathbb{R}^{(m-n)\times n}$ such that $W_2 = U_{\Gamma_n^{\bot}}DV^\top_{\Gamma_n^\bot}+U_{[m-n]}HV^\top$.

\noindent(a1) If $D\neq 0$, then let $i_0$, $j_0$ be any two distinct indices such that $D_{i_0j_0} \neq 0$. Set $\widetilde{W}_2 = U_{\Gamma_n^{\bot}}\widetilde{D}V^\top_{\Gamma_n^\bot}$ with $\widetilde{D}_{ij} = D_{ij}$ if $(i,j) = (i_0, j_0)$, and $\widetilde{D}_{ij} = 0$ otherwise. It is easy to verify that $\widetilde{W}_2\in \mathrm{N}^F_{\mathbb{R}_X(r)}(X ) \subseteq  \mathrm{T}_{\mathcal{L}}(X)$ and hence $\langle \widetilde{W}_2, W_1\rangle =0$. Choose a sequence of matrices
$$X_k = X + t_k \widetilde{W}_2, ~~k = 1, 2, \ldots, $$
with  $t_k>0$ and $\lim\limits_{k\rightarrow \infty} t_k =0$. Note that for each $k$, $\text{rank} (X_k) = s +1 = r$ and $\langle  A^i, X_k\rangle = b_i$ for all $i=1, \ldots, l$. Thus, $X_k \in \mathcal{L}\cap \mathbb{R}(r)$ for all $k$, which shows that $\widetilde{W}_2 \in \mathrm{T}^B_{\mathcal{L}\cap \mathbb{R}(r)}(X)$. By some direct calculation,
$$ \langle \widetilde{W}_2, W\rangle = \langle \widetilde{W}_2, W_1\rangle + \langle \widetilde{W}_2, W_2\rangle = D_{i_0j_0}^2 >0.$$
Since $\mathrm{N}^F_{\mathcal{L}\cap \mathbb{R}(r)}(X)=\left[\mathrm{T}^B_{\mathcal{L}\cap \mathbb{R}(r)}(X)\right]^{\circ}$. This
contradicts to $W\in \mathrm{N}^F_{\mathcal{L}\cap\mathbb{R}(r)}(X )$. Thus, $D =0$.

\noindent(a2) Similarly, if $H\neq 0$, then choose any two indices $i_0\in \{1, \ldots, m-n\}$ and $j_0\in \{1,\ldots, n\}$ satisfying $H_{i_0j_0}\neq 0$. By setting $\widetilde{W}_2 = U_{[m-n]}\widetilde{H}V^\top$ with $\widetilde{H}_{ij} = H_{ij}$ if $(i,j) = (i_0, j_0)$, and $\widetilde{H}_{ij} = 0$ otherwise, we can also verify that $\widetilde{W}_2\in \mathrm{N}^F_{\mathbb{R}_X(r)}(X ) \subseteq  \mathrm{T}_{\mathcal{L}}(X)$. By the same proof as above, we can conclude that $H=0$. Thus, $W_2=0$.

\noindent (b) If $s\leq r-2$, the proof is the same with the one for showing $H=0$ in (a2). This completes the proof.
\end{proof}


The condition $\mathrm{N}^F_{\mathbb{R}_X(r)}(X )\subseteq  \mathrm{T}_{\mathcal{L}}(X)$ in the above theorem seems  strong, but it can be naturally satisfied in some special cases. We see the following example and two corollaries.

\begin{example}Consider the matrix space $\mathbb{R}^{5\times 4}$. Let $e_i\in \mathbb{R}^4$ be the $i$-th column of the identity matrix, $A^1 = [e_1e_1^\top ~0]^\top\in \mathbb{R}^{5\times 4}$, $A^2 = [e_2 e_2^\top~ 0]^\top \in\mathbb{R}^{5\times 4}$, $b = [1,1]^\top$, $r = 3$ and $X_0 = [e_1e_1^\top +e_2e_2^\top ~0]^\top\in \mathbb{R}^{5\times 4}$. By direct calculations, we have
$$\mathcal{L} = \left\{X\in \mathbb{R}^{5\times 4}: X_{11} = X_{22} = 1\right\}, ~~\mathrm{N}_{\mathcal{L}}(X_0) = \left\{\alpha_1 A^1 + \alpha_2 A^2: \alpha_1, \alpha_2\in \mathbb{R}\right\},$$
$$\mathrm{T}_{\mathcal{L}}(X_0)= \left\{X\in \mathbb{R}^{5\times 4}: X_{11} = X_{22} = 0\right\}.$$
Choose $U=I_5,~V=I_4$ to get the SVD of $X_0$. Then $\Gamma = \{1,2\}$, and Assumption \ref{ass2} holds at $X_0$ since $U_{\Gamma}^\top A^1 V_{\Gamma} = \left(
                                                                                                      \begin{array}{cc}
                                                                                                        1 & 0 \\
                                                                                                        0 & 0 \\
                                                                                                      \end{array}
                                                                                                    \right)$,
and $U_{\Gamma}^\top A^2 V_{\Gamma} = \left( \begin{array}{cc}
 0 & 0 \\
  0 & 1 \\
\end{array}
\right)$ which are linearly independent. Moreover, as known from Lemma \ref{NNN},
$$\mathrm{N}^F_{\mathbb{R}_X(r)}(X_0)=\left\{[\beta_1 e_3e_4^\top+\beta_2 e_4e_3^\top~u]^\top: \beta_1, \beta_2\in \mathbb{R}, u\in\mathbb{R}^4 \right\},$$
which is obviously a proper subset of $\mathrm{T}_{\mathcal{L}}(X_0)$.
\end{example}

Let $\mathbb{S}^n$ be the real $n\times n$ symmetric matrix space. It can be verified that the intersection rule as stated in Theorem \ref{NFeq} is also valid in the setting of $\mathbb{S}^n$ equipped with the eigenvalue decomposition, by transferring all the involved notation to the corresponding symmetric variants. Particularly, we have the following special instance where the intersection rule is obtained automatically.

%
\begin{corollary}
Let $X \in\mathcal{L}_{\Delta}\cap\mathbb{S}(r)$ with $$\mathcal{L}_{\Delta}:=\{X\in\mathbb{S}^{n}: \text{tr}(X)=1\},~\mathbb{S}(r):=\{X\in\mathbb{S}^{n}: \text{rank}(X)\leq r\}.$$ Then
\begin{equation}\label{tr}
\mathrm{N}^F_{\mathcal{L}_{\Delta}\cap\mathbb{S}(r)}(X )=\mathrm{N}_{\mathcal{L}}(X )+\mathrm{N}^F_{\mathbb{S}(r)}(X ).
\end{equation}
\end{corollary}
\begin{proof}
Let $
X =U \Lambda(X ){U} ^\top = [U_{\Gamma}~ U_{\Gamma_n^\perp} ]
     \left[
      \begin{array}{cc}
       {\Lambda}_{\Gamma\Gamma }(X ) & 0 \\
       0 & 0 \\
     \end{array}
   \right][ U_{\Gamma } ~ {U}_{\Gamma_n^\perp} ] ^\top$ be its eigenvalue decomposition, where $\Gamma$ is the index set corresponding to those nonzero eigenvalues of $X$. Similar to \eqref{SX}, we introduce the notation
   $$\mathbb{S}_{(X,U)}(r):=\bigcup_{J\in\mathcal{J}}\mathbb{S}_{(X,U)}(J)\subseteq \mathbb{S}(r),$$
   where $\mathbb{S}_{(X,U)}(J):=\left\{U_{J} A U_{J}^\top: A\in \mathbb{S}^{r}\right\}$. To obtain the desired intersection rule, it suffices to show the symmetric variant of Assumption \ref{ass2} and the condition $$\mathrm{N}^F_{\mathbb{S}_{(X,U)}(r)}(X)\subseteq \mathrm{T}_{\mathcal{L}_{\Delta}}(X)~\text{when}~ s<r $$  automatically hold. Denote $\widehat{R}_X=\left[
      \begin{array}{cc}
      U^\top_{\Gamma}IU_{\Gamma} & 0 \\
       0 & 0 \\
     \end{array}
   \right]$. Obviously, the single nonzero matrix $\widehat{R}_X$ is linearly independent. Thus, the symmetric variant of Assumption \ref{ass2} holds at $X$. Meanwhile, the normal cone $\mathrm{N}^F_{\mathbb{S}_{(X,U)}(r)}(X)$ in $\mathbb{S}^n$
takes the form
\begin{equation}\label{TNL}
\mathrm{N}^F_{\mathbb{R}_X(r)}(X)=\left\{ \begin{array}{ll}
           \left\{
  W\in\mathbb{S}^{n}:U_\Gamma^\top W U_\Gamma=0 \right\}, & \text{if}~s=r\\
         \left\{ \begin{array}{l}
         U_{\Gamma_n^\bot}D {U}_{\Gamma_n^\bot}^\top
         \end{array}:\begin{array}{l}
                                                                               \text{diag}(D)=0\\
                                                                               D\in\mathbb{S}^{n-s} \\
                                                               \end{array}
         \right\}, &\text{if}~s=r-1,\\
         \left\{0\right\},&\text{if} ~s\leq r-2.\\
             \end{array}
\right.
\end{equation}
Note that $$\mathrm{T}_{\mathcal{L}_{\Delta}}(X)=\{H\in\mathbb{S}^{n}:\text{tr}(H)=0\}.$$
Together with fact $\text{tr}(  U_{\Gamma_n^\bot}D {U}_{\Gamma_n^\bot}^\top)=\text{tr}(D),$
we have $\mathrm{N}^F_{\mathbb{S}_{(X,U)}(r)}(X)\subseteq \mathrm{T}_{\mathcal{L}_{\Delta}}(X)$, if $s<r$.
Thus, from the symmetric version of Theorem \ref{NFeq}, we can get the desired assertion.
\end{proof}

It is apparent that the sparsity of a vector $x\in\mathbb{R}^n$ coincides with the  rank of the diagonal matrix $\text{Diag}(x)$. The intersection rule in Theorem \ref{NFeq} is also valid in the setting of the Euclidean space of real $n\times n$ diagonal matrices; denoted by $\text{Diag}(\mathbb{R}^n)$ (see, \cite{lewis1996group}), where the SVD is reduce to $\text{Diag}(x)=\text{Diag}(e)\text{Diag}(x)\text{Diag}(e)$ for all $x\in\mathbb{R}^n$. Here $e$ is the all-one vector in $\mathbb{R}^n$. The vector version of the intersection rule is exactly \cite[Corollary 2.10]{Pan2017Optimality}. We propose its matrix version in $\text{Diag}(\mathbb{R}^n)$ to illustrate that Theorem \ref{NFeq} can be regard as a matrix generalization of \cite[Corollary 2.10]{Pan2017Optimality}, where
$\mathrm{N}^F_{\mathbb{R}_X(r)}(X )\subseteq  \mathrm{T}_{\mathcal{L}}(X)~(\text{rank}(X)<s)$ in $\text{Diag}(\mathbb{R}^n)$  automatically holds.

\begin{corollary}(See \cite[Corollary 2.10]{Pan2017Optimality})
Consider the space $\text{Diag}(\mathbb{R}^{n})$. Let $\text{Diag}(x)\in \widehat{\mathcal{L}}\cap \widehat{\mathbb{R}}(r)$ with
$$
\widehat{\mathcal{L}}:=\left\{\text{Diag}(x)\in \text{Diag}(\mathbb{R}^{n}): \langle \text{Diag}(a^i), \text{Diag}(x)\rangle = b_i,  i=1,\ldots, l\right\},
$$
$$
\widehat{\mathbb{R}}(r):=\{\text{Diag}(x)\in \text{Diag}(\mathbb{R}^{n}): ~\text{rank(Diag}(x))\leq r\}.
$$
 If the matrices $\text{Diag}(a^i_\Gamma), i=1,\ldots,l$, are linearly independent, then
 \begin{equation*}
\mathrm{N}^F_{\widehat{\mathcal{L}}\cap \widehat{\mathbb{R}}(r)}(\text{Diag}(x))=\mathrm{N}_{\widehat{\mathcal{L}}}(\text{Diag}(x))+\mathrm{N}^F_{\widehat{\mathbb{R}}(r)}(\text{Diag}(x)).
\end{equation*}
\end{corollary}

\begin{proof} Assumptions  \ref{ass1} and \ref{ass2} are both reduced to the linear independence of $\text{Diag}(a^1_\Gamma),\ldots,\text{Diag}(a^l_\Gamma)$.
Using the counterpart of Theorem \ref{NFeq} in $\text{Diag}(\mathbb{R}^{n})$, it suffices to show that $\mathrm{N}^F_{\widehat{\mathbb{R}}(r)}(\text{Diag}(x))\subseteq \mathrm{T}_{\widehat{\mathcal{L}}}(\text{Diag}(x))$ when $s:=\text{rank}(\text{Diag}(x))<r$. This is trivial since $\mathrm{N}^F_{\widehat{\mathbb{R}}(r)}(\text{Diag}(x))=\{0\}$ whenever $s<r$. This completes the proof.
\end{proof}

\section{Optimality Conditions}\label{sec:Opti}
The optimality analysis, including the first-order and the second-order optimality conditions for the \ref{P}, is proposed in this section, which will provide necessary theoretical fundamentals for handling such a nonconvex discontinuous matrix programming problem. We begin by the introduction of two types of stationary points for the \ref{P}.

For any $X\in \mathbb{R}(r)$ and any $y\in \mathbb{R}^l$, define the Lagrangian function associated with the \ref{P} by
 \begin{equation}\label{LF}
L(X,y)=f(X)+\sum_{i=1}^ly_i[\langle A^i,X\rangle-b_i].
 \end{equation}
 We introduce the following two types of stationary points based on the Lagrangian function.
\begin{definition}\label{def-stat}
 Suppose $\alpha>0$ and $X \in\mathbb{R}(r)$.
\begin{itemize}
 \item[(i)] $X $ is called an $F$-stationary point of the \ref{P} if there exists a vector $y \in\mathbb{R}^l$ such that
\begin{equation}\label{F-sta}
\left\{ \begin{array}{lr}
            \mathcal{A}(X )=b, &  \\
 -\nabla_XL(X ,y )\in \mathrm{N}^{F}_{\mathbb{R}(r)}(X).&\\
\end{array}
\right.
\end{equation}
\item[(ii)] $X $ is called an $\alpha$-stationary point of the \ref{P} if there exists a vector $y \in\mathbb{R}^l$ such that
\begin{equation}
\left\{ \begin{array}{lr}
            \mathcal{A}(X )=b, &  \\
             X \in\Pi_{\mathbb{R}(r)}(X -\alpha\nabla_XL(X ,y )). &
             \end{array}
\right.
\end{equation}

\end{itemize}
\end{definition}

The relationship between these two types of stationary points for the \ref{P} is discussed in the following proposition.

\begin{Proposition}\label{F-M-alpha}
For any given $X \in\mathcal{L}\cap\mathbb{R}(r)$, $y\in\mathbb{R}^l$, and $\alpha>0$, denote
$$\beta=\left\{ \begin{array}{ll}
          \dfrac{{\sigma}_r(X)}{\|\nabla_{X}L(X ;y))\|_2}, & \text{if}~ \|\nabla_{X}L(X ;y))\|_2\neq 0,\\
            \infty, &\text{otherwise}.
             \end{array}
\right.$$
Then, for any $\alpha\in(0,\beta)$, $X$ is an $\alpha$-stationary point of the \ref{P} if and only if $X$ is an $F$-stationary point of the \ref{P}.
%
\end{Proposition}

\begin{proof}
 By mimicking the proof of Theorem 2 in \cite{Li2018Social}, we can obtain that $X$ is an $\alpha$-stationary point of the \ref{P} if and only if
\begin{equation}\label{a}
\nabla_{X}L(X ;y)=\begin{cases}
U_{\Gamma_m^\bot}D V_{\Gamma_n^\bot}^\top,~\mathrm{with}~ \|\nabla_{X}L(X ;y))\|_2\leq\frac{1}{\alpha}{\sigma}_r(X),&\textrm{if~}s=r,\\
0,&\textrm{if~} s<r,\\
\end{cases}
\end{equation}
where $D\in \mathbb{R}^{(m-r)\times (n-r)}$ and $s=\text{rank}(X)$. Together with the expressions of $\mathrm{N}^{F}_{\mathbb{R}(r)}(X)$
as presented in Lemma \ref{NBF}, we can obtain all the desired assertions.
\end{proof}

The first-order optimality conditions in terms of the $F$-stationary point are stated as follows.

\begin{theorem}\label{F-min}
Let $X \in\mathcal{L}\cap\mathbb{R}(r)$ with $\mathrm{rank}(X )=s$.
\begin{itemize}
\item[(i)]  Suppose that  $X $ is a local minimizer of the \ref{P}. If $s=r$ and Assumption \ref{ass1} holds at $X$, or $s<r$ and Assumption \ref{ass2} holds at $X$ with $\mathrm{N}_{\mathbb{R}_X(r)}(X)\subseteq  \mathrm{T}_{\mathcal{L}}(X)$, then $X $ is an $F$-stationary point of the \ref{P}.
\item[(ii)] Suppose that $f$ is a convex function and $X$ is an $F$-stationary point of the \ref{P}. If $s=r$, then $X$ is a global minimizer of the \ref{P} restricted on $\mathbb{R}_X(\Gamma)$; If $s<r$, then $X$ is a global minimizer of the \ref{P}.
\end{itemize}
\end{theorem}
\begin{proof}
 (i) If $X$ is a local  minimizer of the \ref{P}, it follows from the generalized Fermat's theorem and Theorem \ref{NFeq} that
\begin{equation}\label{Fermat} -\nabla f(X )\in \mathrm{N}^F_{\mathbb{R}(r)\cap\mathcal{L}}(X) = \mathrm{N}_{\mathcal{L}}(X) + \mathrm{N}^F_{\mathbb{R}(r)}(X),\end{equation}
if (a) $s=r$ and Assumption \ref{ass1} holds at $X$, or (b) $s<r$, Assumption \ref{ass2} holds at $X$ and $\mathrm{N}_{\mathbb{R}_X(r)}(X)\subseteq  \mathrm{T}_{\mathcal{L}}(X)$. Together with the fact $$\mathrm{N}_{\mathcal{L}}(X) = \left\{\sum\limits_{i=1}^l y^i A^i: y^i\in \mathbb{R}, i=1,\ldots, l\right\},$$
 \eqref{Fermat} indicates that there exists $y\in \mathbb{R}^l$, such that $-\nabla_X L(X,y)\in \mathrm{N}^F_{\mathbb{R}(r)}(X)$. This yields the necessary optimality conditions for the \ref{P} as stated in (i).

(ii) If $s<r$, it follows from \eqref{F-sta} that there exists $y\in\mathbb{R}^l$ such that $\nabla_X L(X,y) =0$ and $L(X,y) = f(X)$. For any feasible solution $Y$ of the \ref{P}, it yields that
$$ f(Y) = L(Y, y) \geq L(X, y)+\langle \nabla_X L(X,y), Y-X\rangle = L(X,y) = f(X),$$
where the inequality follows from the convexity of the function $L(\cdot,y)$ due to the convexity of $f$. Thus we conclude that $X$ is a global solution of the \ref{P}. If $s = r$, then \eqref{F-sta} implies that there exist $y\in \mathbb{R}^l$ and $D\in \mathbb{R}^{(m-r)\times (n-r)}$ such that
\begin{equation}\label{La} \nabla_X L(X,y) = U_{\Gamma_m^\bot}DV^\top_{\Gamma_n^\bot}.
\end{equation} For any $Y\in \mathbb{R}_X(\Gamma)\cap \mathcal{L}$, we can find some matrix $A\in \mathbb{R}^{r \times r}$ such that $Y = U_\Gamma A V_\Gamma^\top$. Thus,
$$ f(Y) = L(Y,y) \geq L(X,y)+\langle \nabla_X L(X,y), Y-X\rangle = L(X,y) = f(X),$$
where the inequality follows from the convexity of $L(\cdot, y)$, and the second equality is from \eqref{La} and $Y-X = U_\Gamma (A-\Sigma(X)) V_\Gamma^\top$. This completes the proof.
\end{proof}

By virtue of the relationship between the $\alpha$-stationary point and the $F$-stationary point as established in Proposition \ref{F-M-alpha}, we have the following theorem by employing Theorem \ref{F-min}.

\begin{theorem}\label{alpha-min}
Let the point $X \in\mathcal{L}\cap\mathbb{R}(r)$with $rank(X)=s$.
\begin{itemize}
\item[(i)] Suppose that $X $ is a local minimizer of the \ref{P}. If $s=r$ and  Assumption \ref{ass1} holds at $X$, or $s<r$ and Assumption \ref{ass2} holds at $X$ with $\mathrm{N}_{\mathbb{R}_X(r)}(X)\subseteq  \mathrm{T}_{\mathcal{L}}(X)$, then there exists $y \in\mathbb{R}^l$ such that, for any $0<\alpha <\beta$, $X$ is an $\alpha$-stationary point of the \ref{P}.
\item[(ii)] Suppose that $f$ is convex and $X$ is an $\alpha$-stationary point of the \ref{P}. If $s=r$, then $X$ is a global minimizer of the \ref{P} restricted on $\mathbb{R}_X(r)$; If $s<r$, then $X$ is a global minimizer of the \ref{P}. 
\end{itemize}
\end{theorem}
\begin{proof}
By applying the equivalence as stated in Proposition \ref{F-M-alpha}, we can obtain the desired assertions readily from Theorem \ref{F-min}.
\end{proof}

To close this section, we study the second-order optimality conditions for the \ref{P}.

\begin{theorem}\label{necessary}
Suppose $f$ is twice continuously differentiable on $\mathbb{R}^{m\times n}$. If $X \in\mathcal{L}\cap\mathbb{R}(r)$ with the SVD as in \eqref{SVD} is a local minimizer of the \ref{P} and Assumption \ref{ass2}  holds at $X $, then
\begin{equation}\label{so-necessary}[\nabla^2f(X )](\Xi,\Xi)\geq 0, ~~~\forall\Xi \in \mathrm{T}_{\mathcal{L}}(X )\cap\mathrm{T}^B_{\mathbb{R}_X(r)}(X )\end{equation}
where $\nabla^2f(X )$ is the Hessian of $f$ at $X $.
\end{theorem}
\begin{proof} It is known from Lemma \ref{TBeq} that $\mathrm{T}_{\mathcal{L}}(X )\cap\mathrm{T}^B_{\mathbb{R}_X(r)}(X ) = \mathrm{T}^B_{{\mathcal{L}}(X )\cap\mathbb{R}_X(r)}(X )$ under Assumption \ref{ass2}. Thus, for any $\Xi \in \mathrm{T}_{\mathcal{L}}(X )\cap\mathrm{T}^B_{\mathbb{R}_X(r)}(X )$, there exist $\left\{X^k\right\}\subseteq {\mathcal{L}}\cap {\mathbb{R}}_X(r)$ and $t_k\downarrow 0$ such that $\lim\limits_{k\rightarrow \infty} \frac{X^k-X}{t_k} = \Xi$. Now we claim that there exists some $y\in\mathbb{R}^l$ such that
\begin{equation}\label{gra-vanish}
\langle \nabla_X L(X,y), X^k-X\rangle = 0, ~~\forall k.
\end{equation}
Case I: If $s = r$, then Theorem \ref{F-min} (i) implies that $X$ is an F-stationary point of the \ref{P} and hence there exists $y\in \mathbb{R}^l$ such that $\nabla_X L(X,y) \in \mathrm{N}^F_{\mathbb{R}(r)}(X) = \mathrm{N}_{\mathbb{R}_s}(X)$. Meanwhile, note that $X^k\in {\mathbb{R}}_X(r) = \mathbb{R}_{X}(\Gamma)$. Thus, \eqref{gra-vanish} follows readily by direct calculation.

\noindent Case II:  If $s<r$, the local optimality of $X$ indicates that \begin{equation}\label{subset}-\nabla f(X) \in \mathrm{N}^F_{\mathcal{L}\cap \mathbb{R}(r)}(X) \subseteq \mathrm{N}_{\mathcal{L}}(X) + \mathrm{N}^F_{\mathbb{R}_X(r)}(X) \end{equation}
where the inclusion is obtained from \eqref{Nequ2}. Thus, there exists $y\in \mathbb{R}^l$ such that
\begin{equation}\label{Lagrange} -\nabla_X L(X,y) \in \mathrm{N}^F_{\mathbb{R}_X(r)}(X) \subseteq \mathrm{N}_{\mathbb{R}_X(J)}(X)= \mathbb{R}_X(J)^\perp, \forall J\in \mathcal{J} \end{equation} where the inclusion is from the observation $\mathrm{N}^F_{\mathbb{R}_X(r)}(X) \subseteq \mathrm{N}^F_{\mathbb{R}_X(J)}(X), \forall J\in \mathcal{J}$. Note that for any given $k$, $X^k \in \mathbb{R}_X(r)$. We can find some $J_k\in \mathcal{J}$ such that $X^k-X \in \mathbb{R}_{X}(J_k)$. Combining with \eqref{Lagrange}, we can also get \eqref{gra-vanish}.

\noindent For any $k$, it then yields that
\begin{eqnarray}
f(X^k) &=& L(X^k, y) \nonumber\\
  &=& L(X, y)+ \frac{1}{2}\nabla_X^2 L(X, y)(X^k-X, X^k-X) + o(\|X^k -X\|^2_F)\nonumber \\
  &=& f(X) + \frac{1}{2}\nabla^2 f(X)(X^k-X, X^k-X) + o(\|X^k -X\|^2_F). \nonumber
\end{eqnarray}
Since $X$ is a local minimizer and $X^k\rightarrow X$, we have $$0\leq \lim\limits_{k\rightarrow \infty} \frac{f(X^k)-f(X)}{t_k^2} = \lim\limits_{k\rightarrow \infty}\frac{1}{2}\nabla^2 f(X)\left(\frac{X^k-X}{t_k}, \frac{X^k-X}{t_k}\right) = \frac{1}{2}\nabla^2 f(X) (\Xi, \Xi).$$
This completes the proof. \end{proof}

\begin{theorem}\label{so-sufficient}
Suppose $f$ is twice continuously differentiable on $\mathbb{R}^{m\times n}$. Let $X\in\mathcal{L}\cap\mathbb{R}(r)$ be an $F$-stationary point of the \ref{P} with its SVD as in \eqref{SVD}. Denotes $s=\mathrm{rank}(X)$. If
\begin{equation}\label{semi-def}[\nabla^2f(X )](\Xi,\Xi)>0,~~~\forall \Xi\in \left(\mathrm{T}_{\mathcal{L}}(X )\cap\mathrm{T}^B_{\mathbb{R}(r)}(X )\right)\setminus\{0\},\end{equation}
holds, we have the following statements.
\begin{itemize}
\item[(i)] If $s = r$, then $X $ is the strictly local optimal solution of the \ref{P} restricted on $\mathbb{R}_X(r)$;
\item[(ii)] If $s<r$, then $X $ is the strictly local optimal solution of the \ref{P}.
\end{itemize}
\end{theorem}
\begin{proof} (i) Consider the case $s = r$. We assume on the contrary that there exists a sequence $\left\{X^k\right\}\subseteq L\cap\mathbb{R}_X(r)$ such that $\lim\limits_{k\rightarrow \infty} X^k = X$, $X^k\neq X$, and $f(X^k)\leq f(X)$ for all $k=1, 2, \ldots$. Denote $\Xi^k: = \frac{X^k-X}{\|X^k-X\|_F}$. The boundedness of the sequence $\left\{\Xi^k\right\}$ admits a convergent subsequence. Without loss of generality, we assume that $\Xi^k\rightarrow \Xi$. Thus, $\Xi \in \mathrm{T}^B_{L\cap\mathbb{R}_X(r)}(X)$ and $\|\Xi\|_F=1$. Since $s = r$, then \begin{equation}\label{equalities}
\mathbb{R}_X(r) = \mathbb{R}_X(\Gamma)  \textrm{~and~}  \mathrm{N}^F_{\mathbb{R}(r)}(X ) = \mathrm{N}_{\mathbb{R}_s}(X).
\end{equation} The first equality in \eqref{equalities} indicates that $X^k-X \in \mathbb{R}_X(\Gamma)$. Since $X$ is an $F$-stationary point of the \ref{P}, combining with the second equality in \eqref{equalities}, we can find some $y\in \mathbb{R}^l$ such that $-\nabla_X L(X,y) \in \mathrm{N}_{\mathbb{R}_s}(X)$. It follows readily that
\begin{equation}\label{zero}
\langle \nabla_X L(X,y), X^k-X\rangle = 0, ~~\forall k.
\end{equation}
Direct calculations then yield
\begin{eqnarray*}
0\geq f(X^k)-f(X) &=& L(X^k, y)-L(X,y) \\
&=&\frac{1}{2}\nabla_X^2 L(X,y)(X^k-X, X^k-X)+o\left(\|X^k-X\|^2_F\right),
\end{eqnarray*}
where the first inequality is from the assumption of $f(X^k)\leq f(X)$, the first equality is from the feasibility of $X^k$ and $X$, and the second equality is from \eqref{zero}.
Thus, $$ 0\geq \lim\limits_{k\rightarrow \infty} \frac{f(X^k)-f(X)}{\|X^k-X\|^2_F} = \lim\limits_{k\rightarrow \infty} \frac{1}{2}\nabla_X^2 L(X,y)(\Xi^k, \Xi^k) = \frac{1}{2}\nabla_X^2 L(X,y)(\Xi, \Xi).$$
Note that $\nabla_X^2 L(X,y) = \nabla^2 f(X)$ and $\Xi\in \mathrm{T}^B_{L\cap\mathbb{R}_X(r)}(X) \subseteq \mathrm{T}^B_{L\cap\mathbb{R}(r)}(X) \subseteq \mathrm{T}_{L}(X)\cap \mathrm{T}^B_{\mathbb{R}(r)}(X)$. This arrives at a contradiction to \eqref{semi-def}. Thus, $X$ is a strictly local optimal solution of the \ref{P} restricted on $\mathbb{R}_X(r)$.

(ii) If $r<s$, the $F$-stationary point $X$ allows us to find some $y\in \mathbb{R}^l$ such that $\nabla_X L(X,y) =0$. Thus \eqref{zero} holds automatically. By mimicking the proof as above in (i), we can also obtain the desired strictly local optimality at $X$ for the \ref{P}. The proof is completed.
\end{proof}

\section{Applications}
Recently there has been a surge of interest in low-rank matrix optimization subject to some problem-specific constraints often characterized as an affine set. In this section, we illustrate how the results presented in Section 3 can be applied to the real world problems.


\noindent{\bf Hankel matrix approximation}\cite{grussler2018low-rank,qi2018a}.
Low-rank approximation has appeared in data analysis, system identification, control and so on(see, e.g.\cite{fazel2002matrix,Defeng2013hankel,Ankelhed}). For example, in the automatic control, the rank of a Hankel matrix determines the order of a linear dynamical system, and finding a low-rank Hankel matrix approximation can be formulated as
\begin{equation}\label{HMA}
  \begin{aligned}
\min\limits_{X\in\mathbb{R}^{m\times n}}&~~\dfrac{1}{2}\|H-X\|_F^2\\
s.t.&~~ X\in\mathcal{H},\\
 &~~ \mathrm{rank}(X)\leq r,
  \end{aligned}
\end{equation}
where
\begin{eqnarray*}
\mathcal{H}:&=&\{X\in\mathbb{R}^{m\times n}:X ~\text{is}~ \text{Hankel}\}\\
&=&\{X\in\mathbb{R}^{m\times n}:\langle E_{kj}-E_{k-1,j+1}, X\rangle=0,~\forall ~k=2,\ldots, m,~\forall~ j=1,\ldots, n-1\}.
\end{eqnarray*}
Here $E_{kj}$ denotes a matrix in which the $(k,j)$-th element is 1, and the remaining elements are $0$.
The Lagrangian function associated with the \eqref{HMA} takes the form of
$$L(X,y)=\dfrac{1}{2}\|H-X\|_F^2+\sum_{i=1}^{(m-1)(n-1)} y_i\langle A^i, X\rangle$$
where $A^i=E_{kj}-E_{k-1,j+1}\in\mathbb{R}^{m\times n}$ and $y_i\in\mathbb{R}$ is the Lagrangian multiplier corresponding to the equality constraint for $i=1,\ldots, (m-1)(n-1)$. It is easy to show the gradient of $L(X,y)$ with respect to $X$ is
$$\nabla_XL(X,y)=X-H+\sum_{i=1}^{(m-1)(n-1)} y_i A^i,$$
where
\begin{eqnarray*}
\sum_{i=1}^{(m-1)(n-1)} y_i A^i&=&\left[
 \begin{array}{cccccc}
    a_1& -b_1+a_2& -b_2+a_3 & \ldots & -b_{n-2}+a_{n-1} & -b_{n-1} \\
     \end{array}
  \right]\\
  &=&\left[
  \begin{array}{cc}
  I_m & -I_m\\
 \end{array}
 \right]
 \left[
  \begin{array}{ccccc}
  a_1 & a_2& \ldots & a_{n-1} &0\\
    0& b_1& \ldots & b_{n-2} &b_{n-1}\\
     \end{array}
 \right]
  \end{eqnarray*}
for
$a_j=\sum_{i=0}^{n-2}y_{i(n-1)+j}e_{i+2}\in\mathbb{R}^m$ and  $b_j=\sum_{i=0}^{n-2}y_{i(n-1)+j}e_{i+1}\in\mathbb{R}^m$.
We take $m=3$ and $n=3$ and $H=\left[
 \begin{array}{ccc}
     e_2& e_1& e_3 \\
    \end{array}
 \right]$
as an example.~Then \eqref{HMA} can be reformulated as
\begin{equation}\label{HMA3}
  \begin{aligned}
\min\limits_{X\in\mathbb{R}^{3\times 3}}&~~\|H-X\|_F^2\\
s.t.&~~\langle A^i, X\rangle=0, ~\forall i=1,\ldots,4,\\
 &~~ \mathrm{rank}(X)\leq 2,
  \end{aligned}
\end{equation}
where
$$A^1=\left[
 \begin{array}{ccc}
      e_2 & -e_1 & 0 \\
     \end{array}
  \right],~~~~~A^2=\left[
   \begin{array}{ccc}
      0 & e_2 & -e_1 \\
     \end{array}
  \right],$$
  $$A^3=\left[
     \begin{array}{ccc}
    e_3 & -e_2 & 0\\
     \end{array}
  \right],
 ~~~~~A^4=\left[
     \begin{array}{ccc}
      0 & e_3 & -e_2 \\
     \end{array}
  \right].$$
  Clearly, the objective function is convex, and
  $$-\nabla_X L(X,y)=H-X-\left[
 \begin{array}{ccc}
      a_1 & -b_1+a_2& -b_2 \\
     \end{array}
  \right].$$
  Consider the matrix $\overline{X}=\left[
 \begin{array}{ccc}
     e_2& e_1& 0 \\
     \end{array}
 \right]$ with $\Gamma=\{1, 2\}$, $\overline{U}=[e_1~ e_2~ e_3]$ and $\overline{V}=[e_2~ e_1~ e_3]$. Then, by \eqref{Ti}, we have
 $$T_{\overline{X}}^1=\left[
 \begin{array}{ccc}
     -e_1& e_2& 0 \\
     \end{array}
 \right],~~T_{\overline{X}}^2=\left[
 \begin{array}{ccc}
     e_2& 0& -e_1 \\
     \end{array}
 \right],$$
 $$T_{\overline{X}}^3=\left[
 \begin{array}{ccc}
     -e_2& e_3& 0 \\
     \end{array}
 \right],~~T_{\overline{X}}^4=\left[
 \begin{array}{ccc}
     e_3& 0& -e_2 \\
     \end{array}
 \right],$$
 which implies that Assumption \ref{ass1} holds. From Lemma \ref{NBF}, we obtain
$$\mathrm{N}^F_{\mathbb{R}(2)}(\overline{X})=\mathrm{N}_{\mathbb{R}_2}(\overline{X})=\{\beta e_3e_3^\top: \beta\in\mathbb{R}\}.$$
  Direct calculation gives
  \begin{eqnarray*}
  -\nabla_X L(\overline{X},y)&=&\left[
 \begin{array}{ccc}
     0 & 0 & e_3 \\
     \end{array}
  \right]\\
  &&-\left[
 \begin{array}{ccc}
      y_1e_2+y_3e_3 & -y_1e_1+y_2e_2-y_3e_2+y_4e_3 & -y_2e_1-y_4e_2 \\
     \end{array}
  \right]\\
  &=&\left[
      \begin{array}{ccc}
   0& y_1 & y_2  \\
   -y_1& y_3 -y_2 &y_4\\
    -y_3 & -y_4 & 1 \\
     \end{array}
   \right].
  \end{eqnarray*}
There exist $y_i=0$ for $i=1,2,3,4$ such that $-\nabla_X L(\overline{X},y)=e_3e_3^\top\in\mathrm{N}^F_{\mathbb{R}(2)}(\overline{X})$. Then, $\overline{X}$ is an $F$-stationary point of \eqref{HMA3}. Thus, from Theorem \ref{F-min} (ii), we obtain that the $\overline{X}$ is a global minimizer of \eqref{HMA3} restricted on $\mathbb{R}_{\overline{X}}(\Gamma)$.\\

\noindent{\bf Low-Rank representation over the manifold}\cite{Yin2015Nonlinear, YFu2105, Wang2015Kernelized}.
 ~Low-rank representation (LRR) has attracted great interest in computer vision, pattern recognition and signal processing (see, e.g., \cite{delvaux2007givens,wright2010sparse,xue2018robust}). However, in many computer vision applications, data often originate from a manifold, which is equipped with some Riemannian geometry.
To address this problem, an LRR over manifold is
proposed, and it can be formulated as
\begin{equation}\label{cLRR}
  \begin{aligned}
\min\limits_{W\in\mathbb{R}^{N\times N}}&~~\dfrac{1}{2}\sum_{i=1}^N w_iB^iw_i^\top\\
s.t.&~~ \sum_{j=1}^N W_{ij}=1, ~i=1,\ldots, N,\\
 &~~ \mathrm{rank}(W)\leq r,
  \end{aligned}
\end{equation}
where $r>1$ and $B^i\in\mathbb{R}^{N\times N}$, $w_i$ is the $i$-th row of matrix $W\in\mathbb{R}^{N\times N}$. The Lagrangian function of \eqref{cLRR} is given by
$$L(W,y)=\dfrac{1}{2}\sum_{i=1}^N w_iB^iw_i^\top+\sum_{i=1}^N y_i[\langle E^i, W\rangle -1],$$
where $y=(y_1,\ldots, y_N)$ is the Lagrangian multiplier corresponding to the equality constraint, and $E^i\in\mathbb{R}^{N\times N}$ denote the matrix with all components $1$ in the $i$-th row, and $0$ otherwise. It is easy to show the gradient of $L(W,y)$ with respect to $W$ is
$$\nabla_WL(W,y)=\left[
     \begin{array}{ccc}
      w_1B^1 \\
      \vdots\\
     w_NB^N \\
     \end{array}
  \right]+ye^T,$$
where $e\in\mathbb{R}^N$ denotes the vector with all components one. We take $B^i=I_N$ for $i=1,\ldots, N$ as an example. Clearly, the objective function is convex, and the gradient
$\nabla_WL(W,y)=W+y e^T.$
Consider the matrix $\overline{W}=\dfrac{1}{N}E$, where $E\in\mathbb{R}^{N\times N}$ denotes matrix with all components one. Direct calculation gives  $\nabla_WL(\overline{W},y)=\dfrac{1}{N}E-y e^T$.
 There exist $y_i=-\dfrac{1}{N}$ for $i=1,\ldots, N$ such that $\nabla_WL(\overline{W},y)=0$. Then, $\overline{W}$ with $\text{rank}(\overline{W})=1<r$ is an $F$-stationary point of \eqref{cLRR}. Thus, from Theorem \ref{F-min} (ii), we obtain that $\overline{W}$ is a global minimizer of \eqref{cLRR}. 

\section{Conclusions}
This paper is concerned with the low-rank matrix optimization problem whose feasible set is the intersection of the rank constraint set and an affine set. We have explored the intersection rule of Fr{\'e}chet normal cone to the feasible set of the \ref{P} relying on the linear independence assumption. With the help of this result, we have derived the first-order necessary, and the first-order sufficient optimality conditions for the \ref{P}. Moreover, the second-order necessary, and the second-order sufficient optimality conditions are also presented based on the Bouligand tangent cone. To illustrate the effectiveness of these optimality conditions, two specific applications of the \ref{P} are discussed. To the best of our knowledge, this paper is the first one to touch the optimality conditions for the original low-rank optimization problem \ref{P}. These proposed results not only enrich the optimality theory of matrix optimization, but also facilitate the algorithm design for the \ref{P}. In particular, the characterization of an $\alpha$-stationary point presents a much easier way to design numerical algorithms to search. There remain some issues worth pursuing further research: The optimality conditions to nonlinear equality and inequality constraints need to be derived and discussed, and effective algorithms for solving the \ref{P} need to be developed.
\section*{Acknowledgments}
The work was supported in part by the National Natural Science Foundation of China (Grants 11971052, 11771038) and Beijing Natural Science Foundation (Z190002).

\bibliographystyle{plain}
\bibliography{sdp}

\end{document}